\newcounter{myctr}

\documentclass{amsart}
\usepackage{url}
\usepackage{verbatim}
\usepackage{xspace}
\usepackage{amsmath}
\usepackage{amssymb}
\usepackage{color}
\usepackage{graphicx}

\newtheorem{theorem}{Theorem}[section]
\newtheorem{lemma}[theorem]{Lemma}
\newtheorem{proposition}[theorem]{Proposition}
\newtheorem{corollary}[theorem]{corollary}

\theoremstyle{definition}
\newtheorem{definition}[theorem]{Definition}
\newtheorem{example}[theorem]{Example}
\newtheorem{remark}[theorem]{Remark}

\graphicspath{{./}{./figs/}}
\def\gds{GDS\xspace}
\def\sds{SDS\xspace}
\def\gca{GCA\xspace}
\def\etgds{ET-GDS\xspace}
\def\etsds{ET-SDS\xspace}
\def\etgca{ET-GCA\xspace}
\def\F{\mathbf{F}}

\def\Per{\mathrm{Per}}
\def\Fix{\mathrm{Fix}}
\def\vset{\mathrm{v}}
\def\eset{\mathrm{e}}
\def\N{\mathbb{N}}
\def\Circ{\mathrm{C}}
\def\Path{\mathrm{P}}
\def\Fib#1{\mathrm{Fib}_{#1}}
\def\fib{\mathrm{Fib}}
\def\Luc{\mathrm{Luc}}
\def\supp{\ensuremath{\operatorname{supp}}}

\def\card#1{|#1|}
\definecolor{red}{rgb}{1.0,0.0,0.0}

\begin{document}
\makeatletter
\def\@biblabel#1{[#1]}
\makeatother
\markboth
{L. Chang, J. Cochran, H.~S. Mortveit, S. Raval, M. Schroeder}
{Adaptive Complex Contagions and Threshold Dynamical Systems}
%
\title[Adaptive Complex Contagions and Threshold Dynamical Systems]{ADAPTIVE COMPLEX CONTAGIONS AND THRESHOLD DYNAMICAL SYSTEMS}
\author[L. Chang]{LEON CHANG}
\address{Department of Applied Physics \& Applied Mathematics,
  Columbia University\\
lc2585@columbia.edu}
\author[J. Cochran]{JEFFREY COCHRAN}
\address{Department of Mathematics and Statistics, Georgetown University\\
jdc62@georgetown.edu}
\author[H.S. Mortveit]{HENNING S.~MORTVEIT}
\address{Department of Mathematics \& NDSSL,
Virginia Tech\\
hmortvei@vbi.vt.edu}
\author[S. Raval]{SIDDHARTH RAVAL}
\address{Department of Mathematics, Reed College\\
ravals@reed.edu}
\author[M. Schroeder]{MATTHEW SCHROEDER}
\address{Applied Mathematics, Geneva College\\
matthew.schroeder@geneva.edu}
\maketitle
\begin{abstract}
A broad range of nonlinear processes over networks are governed by
threshold dynamics. So far, existing mathematical theory
characterizing the behavior of such systems has largely been concerned
with the case where the thresholds are \emph{static}. In this paper we
extend current theory of finite dynamical systems to cover
\emph{dynamic thresholds}.  Three classes of parallel and sequential
dynamic threshold systems are introduced and analyzed. Our main
result, which is a complete characterization of their attractor
structures, show that sequential systems may only have fixed points as
limit sets whereas parallel systems may only have period orbits of
size at most two as limit sets. The attractor states are characterized
for general graphs and enumerated in the special case of paths and
cycle graphs; a computational algorithm is outlined for determining
the number of fixed points over a tree. We expect our results to be
relevant for modeling a broad class of biological, behavioral and
socio-technical systems where adaptive behavior is central.
\end{abstract}
\maketitle

\section{Introduction}
\label{sec:intro}

Many biological, social and technical systems can be described as
\emph{dynamical processes over graphs}. A specific example is the
spread of influenza in a human population~\cite{Eubank:04}. Here, the
vertices of a graph are used to represent the individuals of the
population with edges representing their contacts. Each vertex is
assigned a \emph{vertex state} that captures the particular
individual's disease status and possibly other relevant factors such
as behavioral or immunological characteristics. Based on the current
state and the health status of the neighbors, a \emph{vertex function}
governs the evolution of this particular individual's health state
with time. The application of the entire ensemble of vertex functions
determines how the global disease dynamics of the population evolves
with time. Similar examples include social \& behavioral
systems~\cite{Granovetter:78,Centola:07,Centola:09}, spread of malware
on wired \& wireless networks~\cite{Channakeshava:11}, and gene
prediction~\cite{Karaoz:04}.

\emph{Threshold functions} are widely used to capture the \emph{local
  dynamics} of systems such as those above. In its basic form, a
threshold function with threshold~$k$ is a Boolean function that
returns the value~$1$ (or true) if~$k$ or more of its binary inputs
are~$1$, and returns~$0$ (or false) otherwise.
The prominent role of threshold functions in modeling and applications makes
it desirable to have a solid understanding of their properties. Existing
theoretical work (see,
e.g.,~\cite{Barrett:06a,Mortveit:07,Kuhlman:12,Goles:80}) has largely been
concerned with the case where each vertex threshold~$k$ is fixed. While this
may be adequate in many situations, there are systems where the threshold
values naturally change with time. Again, taking influenza as an example, the
immune system typically receives a boost for the particular flu virus strain
after a clinical episode, effectively increasing one's threshold value for
falling sick.
For addictive behaviors such as smoking, the threshold (related to
peer-pressure, for example) for re-smoking after quitting may often be
less than the initial threshold.
For more complex diseases such as malaria, which involves acquired immunity,
one may see both drops and rises in threshold values: increased immunity is
developed upon exposure whereas a loss of immunity occurs through times of no
exposure~\cite{Smith:08}.
Using the framework of \emph{graph dynamical systems} (\gds), see for
example~\cite{Mortveit:01a,Macauley:09a,Mortveit:07,Macauley:11c,Macauley:10b,Laubenbacher:01a},
we introduce three classes of dynamic threshold function that target
the three cases described above and other complex contagions
(see,~e.g.~\cite{Centola:07,Granovetter:78}).

The three dynamic threshold functions we consider differ from standard
threshold functions by allowing the threshold of each vertex to change when its
state transitions from $0$ to $1$ or from $1$ to $0$. Specifically, for the
\emph{increasing} (resp. \emph{decreasing}) threshold function the vertex
threshold increases (resp. decreases) by $1$ under the $0 \longrightarrow 1$
(resp. $1 \longrightarrow 0$) transition.  The \emph{mixed} threshold function
combines the behaviors of the increasing and decreasing threshold
functions.

{\bfseries Paper outline.}  In Section~\ref{sec:def} we introduce
necessary background and terminology for graph dynamical systems along
with definitions for \emph{increasing}, \emph{decreasing} and
\emph{mixed} threshold graph dynamical systems (\etgds).  Our main
result is presented in Section~\ref{sec:limitsets} and states that
increasing, decreasing and mixed threshold \emph{sequential} graph
dynamical systems only have fixed points as attractors. Moreover, for
the synchronous case with increasing, decreasing and mixed threshold
functions periodic orbits have length at most~$2$. Our approaches use
several techniques. Most notably, the argument for the sequential
mixed threshold case uses an extension of a potential function
argument developed by Marathe et al in~\cite{Barrett:06a}.  The
parallel mixed threshold argument extends the classical proof by Goles
and Olivos. Whereas their proof was developed for neural networks, we
limit the statement of our main proof to the case where all edge
weights are~$1$. However, our proof for the synchronous mixed
threshold case is given for neural networks.  We also prove that
increasing and decreasing threshold systems are topologically
conjugated both in the synchronous and asynchronous case, and we
demonstrate that the six classes of dynamic threshold systems have a
common set up fixed points.  In light of this fact, our next class of
results in~Section~\ref{sec:enumeration} are on enumeration of this
common set of fixed points. Specifically, we consider the path and
circle graphs since these are building blocks of more general graphs.
Here we also present scaling properties for the number of fixed points
as a function of graph size. Our final result is an algorithm for
determining the number fixed points for dynamic threshold \gds when
the graph is a tree.
We remark that the problem of enumerating and finding fixed points is, in
general, NP-complete~\cite{Barrett:01g}.


\section{Definitions and Terminology}
\label{sec:def}

First, we define \emph{sequential dynamical systems} (\sds) and
\emph{generalized cellular automata} (\gca) which are both special
cases of \emph{graph dynamical systems} (\gds). We largely follow the
notation in~\cite{Macauley:09a}.
Let $X$ be a simple \emph{graph} on $\card{X} = n$ vertices and $K$ a
finite set. Associated to each vertex $v \in \{1,2,\ldots n\}$ is a
\emph{vertex state} $x_v \in K$. We write $x=(x_1,\ldots,x_n)$ for the
\emph{system state}.  Next, we let $n[v]$ be the sequence of vertices
contained in the $1$-neighborhood of~$v$ ordered in increasing order
with $v$ included. Also, let~$x[v]$ denote the restriction of the
\emph{system state} to~$n[v]$.
For each vertex~$v$ we have a \emph{vertex function}
\begin{eqnarray*}
  f_v \colon K^{d(v)+1} \longrightarrow K\;,
\end{eqnarray*}
where $d(v)$ is the \emph{degree} of $v$, and an \emph{$X$-local function} $F_v
\colon K^n \longrightarrow K^n$ given by
\begin{equation*}
  F_{v}(x) = (x_{1},\ldots, f_{v}(x[v]), \ldots, x_{n}) \;.
\end{equation*}
The vertex function $f_v$ takes the state of vertex $v$ and its
neighbors at time $t$ as input and computes the state for vertex~$v$
at time $t+1$.  The choice for when to use vertex functions or
$X$-local functions depends on the context; $X$-local functions are
typically used for sequential systems since these functions can be
composed. Some of the concepts above are illustrated in
Figure~\ref{fig:gds}.
\begin{figure}[ht]
\centerline{\includegraphics[]{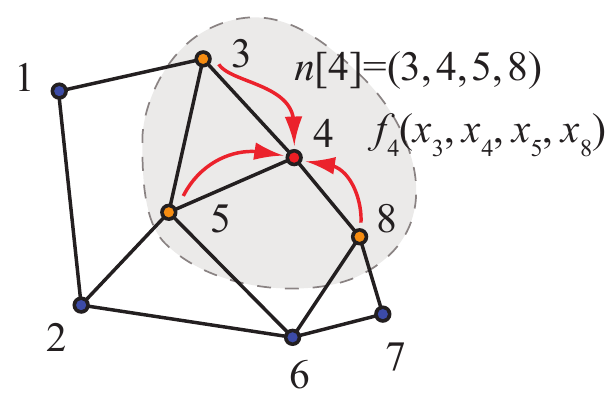}}
\caption{In the figure we have $n[4] = (3,4,5,8)$, and $x[4] = (x_3,
  x_4, x_5, x_8)$. The vertex function $f_4$ takes $x[4]$ as input to
  compute the state of vertex $4$ at the next time step.}
\label{fig:gds}
\end{figure}

Let $S_X$ be the symmetric group over the vertex set of $X$. Here a
permutation $\pi = (\pi_1,\dots,\pi_n) \in S_X$ induces an ordering
$<_\pi$ on the vertex set by $\pi_k <_\pi \pi_l$ if and only if $k<l$,
that is, $i<_\pi j$ if $i$ occurs before $j$ in $\pi$.
\begin{definition}[\sds, \gca]
Let $X$, $K$, $(f_v)_{v=1}^n$ and $\pi\in S_X$ be as above. The
corresponding sequential dynamical system and generalized cellular
automaton maps $\F_\pi$, $\F \colon K^n \longrightarrow K^n$ are
defined by
\begin{equation*}
  \F_\pi = F_{\pi_n}\circ F_{\pi_{n-1}} \circ\cdots\circ F_{\pi_1}
\end{equation*}
and
\begin{equation*}
  \F(x_1,\ldots, x_n) = \bigl(f_1(x[1]), \ldots, f_b(x[n])\bigr)\;,
\end{equation*}
respectively.
\end{definition}
Applying $\F_\pi$ or $\F$ to a system state $x\in K^n$ is called a
\emph{system update}, whereas applying $F_v$ to $x$ is called a
\emph{vertex update}. Here the \gca map $\F$ corresponds to
\emph{synchronous/parallel} updating of vertex states, and the \sds
map $\F_\pi$ corresponds to \emph{asynchronous/sequential} updating
using the sequence~$\pi$.
The \emph{phase space} of the dynamical system with map $\phi\colon
K^n \longrightarrow K^n$ is the directed graph $\Gamma(\phi)$ with
vertex set $K^n$ and edge set $\{(x,\phi(x) | x\in K^n \}$. A vertex
or state on a cycle in $\Gamma(\phi)$ is a \emph{periodic point} of
$\phi$; a state on a cycle of length one is a \emph{fixed point}. All
other states are \emph{transient states}. We denote the sets of
periodic points and fixed points of $\phi$ by $\Per(\phi)$ and
$\Fix(\phi)$, respectively. These points and their transitions
represent the long-term behavior of the dynamical system~$\phi$.

In this paper we are concerned with generalizations of the case where the
vertex functions are threshold functions. With state space $K = \{0,1\}$, a
\emph{standard threshold function} $\tau_{k,m} \colon K^{m} \longrightarrow K$ is
defined by
\begin{equation*}
\tau_{k,m}(x_1,\ldots,x_m) =
\begin{cases}
 1, & \sigma(x_1,\ldots,x_m)\ge k \\
 0, & \textrm{else,}
\end{cases}
\end{equation*}
where $\sigma(x_1,\ldots,x_m) = \card{ \{ i \mid x_i = 1\} }$.  To
each vertex $v$ we associate a threshold value $k_v \in
\mathbb{N}$. In contrast to standard threshold systems, we will not
require that the thresholds $k_v$ to be fixed, and will incorporate the
vertex threshold in the vertex state. Thus, for each vertex $v$ we
have an \emph{extended vertex state} $s_v = (x_v, k_v) \in K_v =
\{0,1\} \times D_v$ where $D_v = \{1,2,\ldots,d(v)+1\}$. In our case,
the system state is therefore an element of $\mathcal{S} =
\prod_{i=1}^n K_i$. Note that we have excluded the constant functions
(zero and one) through the choice of thresholds in the set $D_v$.
\begin{definition} An extended threshold graph dynamical system
  (\etgds) is a \gds where each vertex function $f_v \colon
  \prod_{i\in n[v]} K_i \longrightarrow K_v$ is given by
\begin{equation}
\label{eq:genthreshold}
f_v( s[v] ) =  (\tau_{k_v, d(v)+1}( x[v] ), g_v(x[v], k_v) ) \;,
\end{equation}
where $g_v$ is some function governing the evolution of the vertex
threshold. 
\end{definition}
Although a slight abuse of terminology, we will occasionally call the
tuple $x = (x_1,\dots,x_n)$ the \emph{system state}, and the tuple of
thresholds $k = (k_1,\dots,k_n)$ the \emph{system threshold state}
($k$). Wherever the context warrants a distinction, we will refer to
the tuple of extended vertex states $s = (s_1,\dots,s_n) =
\bigl((x_1,k_1), \dots, (x_n,k_n)\bigr)$ as the \emph{extended system
  state}.

In this paper, we consider three specific classes of \etgds corresponding to
three choices of the function $g_v$ in~\eqref{eq:genthreshold}. For the
\emph{increasing threshold vertex function} the function $g_v$ is given by
\begin{equation}
  \label{eq:incr}
  g^\uparrow_v( x[v], k_v ) =
  \begin{cases}
    k_v + 1, &   \text{if } x_v = 0              \land
    \sigma(x[v]) \ge k_v 
    \text{, and} \\
    k_v, & \text{otherwise.}
  \end{cases}
\end{equation}
Thus, for $g^\uparrow_v$ the threshold $k_v$ increases by~$1$ every time $x_v$
transitions from $0$ to $1$.
We denote the corresponding \etsds and \etgca by $\F^\uparrow_\pi$ and
$\F^\uparrow$.

Similarly, we define \emph{decreasing \etgds} by letting the vertex
threshold decrease by one whenever $x_v$ is mapped from~$1$ to~$0$. We
denote the corresponding maps by $g^\downarrow_v$, $f^\downarrow_v$,
$F^\downarrow_v$, $\F^\downarrow_\pi$, and $\F^\downarrow$.
Finally, a \emph{mixed \etgds} combines the definitions of
$f^\uparrow$ and $f^\downarrow$. Specifically, the $x_v$ component of
the state is mapped as before using the function $\tau$, whereas the
threshold map $g^\updownarrow_v$ will increase (resp. decrease) $k_v$
by~$1$ if $x_v = 0$ (resp. $1$) and $\sigma(x[v]) \ge k_v$ (resp. $<
k_v$), and will map $k_v$ identically in the remaining case. For the
mixed threshold function, the maps are denoted by $g^\updownarrow$,
$f^\updownarrow$, $\F^\updownarrow_v$, $\F^\updownarrow_\pi$ and
$\F^\updownarrow$.
\begin{example}
\label{ex:1}
Consider the case where the graph $X$ is a $2$-path with \etgds map
$\F^\uparrow_\pi$ and update sequence $\pi = (1,2)$. The~$(2^2)^2 = 16$ states
and their transitions are shown in the phase space of
$\Gamma(\F^\uparrow_\pi)$, see Fig.~\ref{fig:ex1}. Clearly, there are~10 fixed
points and six transient states.
\begin{figure}[ht]
\centerline{\includegraphics[scale=0.36]{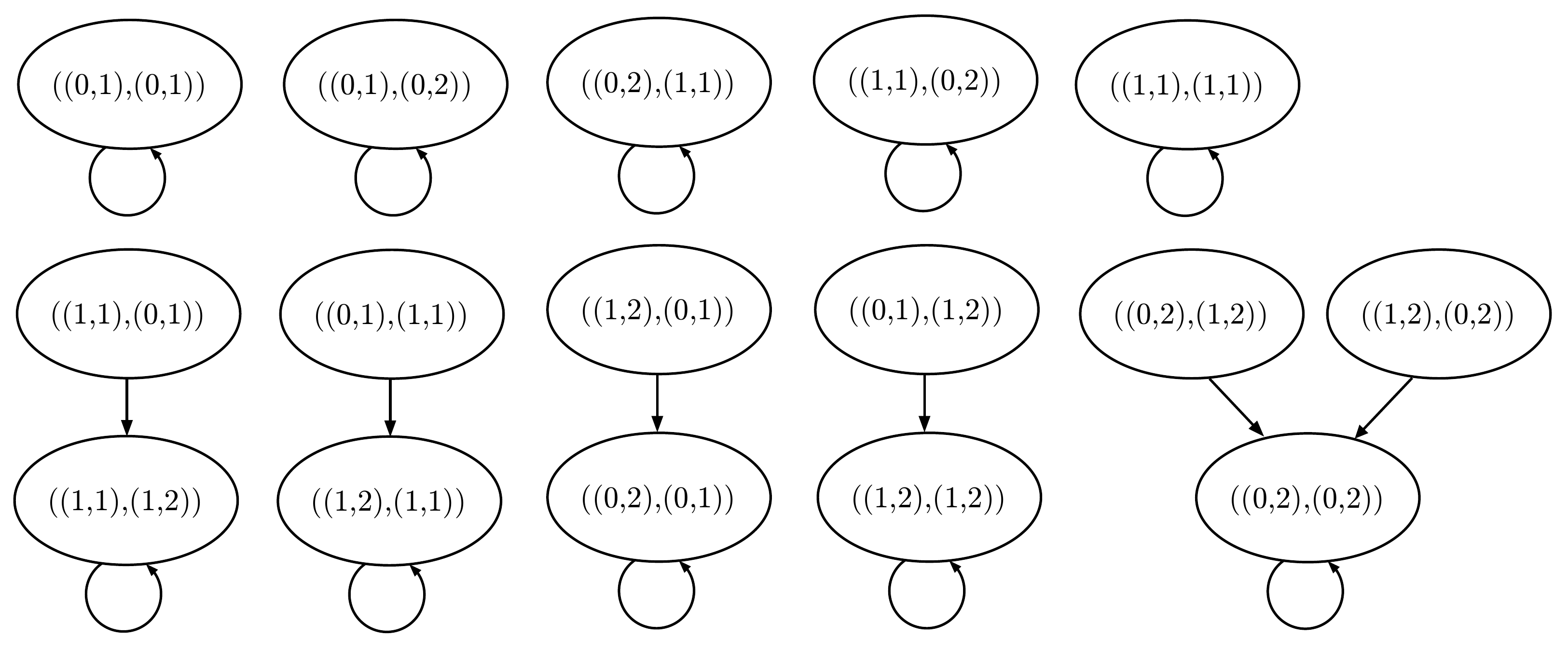}}
\caption{The phase space of Example~\ref{ex:1}. Each extended system
  state $s = \bigl((x_1,k_1), (x_2,k_2)\bigr)$ is outlined with an oval.}
\label{fig:ex1}
\end{figure}
\end{example}


\section{Characterizations of Limit Sets for \etgds}
\label{sec:limitsets}

In this section we classify the periodic orbit structure of \etgds
by proving the following result through a sequence of propositions and
lemmas. 
\begin{theorem}
\label{thm:main}
The maximal size of a periodic orbit for increasing, decreasing and mixed
\etsds maps is~$1$. For \etgca maps, the maximal size is~$2$.
\end{theorem}
Additionally, we relate the phase spaces of increasing and decreasing
\etgds and show that their maps are topologically conjugated. We start
by recalling the results for standard threshold systems. The proof of
the statements in following lemma can be found in~\cite{Barrett:06a}
for the sequential case and in~\cite{Goles:80} for the parallel case.
\begin{lemma}
\label{lem:classic}
A standard threshold \sds (resp. \gca) map has no periodic orbit of
size~$\ge 2$ (resp.~$\ge 3$).
\end{lemma}
We can now state our result for increasing threshold \gds maps.
\begin{proposition}
\label{prop:incr}
For any graph $X$ and any update sequence $\pi\in S_X$ the \etsds map
$\F_\pi^\uparrow$ has no periodic orbit of length~$\ge 2$. Similarly, the
\etgca map $\F^\uparrow$ has no periodic orbit of length~$\ge 3$.
\end{proposition}
\begin{proof}
Choose an arbitrary state $s = (x,k) \in \mathcal{S}$ and consider the orbit
$\mathcal{O}(s)$ of $\F^\uparrow_\pi$ starting at $s$. By definition, it follows
that each component of $k$ is non-decreasing and bounded along
$\mathcal{O}(s)$. Since the state space is finite, there exists an integer
$r\ge 0$ such that each $k$-component of $F^{\uparrow (u)}_\pi(s)$
is fixed for $u\ge r$. Consequently, for $u\ge r$ the dynamics of
$\F^\uparrow_\pi$ coincides with a standard threshold \sds over $X$, and,
using Lemma~\ref{lem:classic}, we conclude that $s$ is eventually fixed. Since
$s$ was arbitrary, the first statement follows.  The proof for increasing
threshold \gca maps is completely analogous and is therefore omitted.
\end{proof}
A similar proof can be constructed for the decreasing threshold \sds
and \gca, however a stronger statement is possible: increasing and
decreasing threshold \gds maps are in fact topologically conjugated.
\begin{proposition}
\label{prop:conj}
For any graph $X$ and any update sequence $\pi\in S_X$ there exists a
bijection $\psi \colon \mathcal{S} \longrightarrow \mathcal{S}$ such
that
\begin{equation}
\label{eq:conj}
\psi \circ \F^\uparrow_\pi = \F^\downarrow_\pi \circ \psi\;,
\end{equation}
where $\mathcal{S} = \prod_i K_i$.
\end{proposition}
In other words, the maps~$\F^\uparrow_\pi$ and~$\F^\downarrow_\pi$ are
topologically conjugated.  In other words, their phase spaces are
isomorphic as directed graphs.
\begin{proof}
We prove this by constructing the bijection $\psi$ in~\eqref{eq:conj}
directly. To this end, let $\psi = (\psi_1,\ldots,\psi_n) \colon \mathcal{S}
\longrightarrow \mathcal{S}$ be defined coordinate-wise by setting $\psi_i
\colon \{0,1\} \times D_i \longrightarrow \{0,1\} \times D_i$ equal to
\begin{equation}
  \label{eq:psiconj}
  \psi_i(x_i, k_i) = ( x_i + 1 \text{ mod } 2, d(v_i) - k_i + 2 )\;.
\end{equation}
We claim that $\psi$ is its own inverse. To see this, let $(x_i,k_i)$
denote an arbitrary vertex state for vertex~$i$. Then
\begin{align*}
  (\psi_i \circ \psi_i)(x_i,k_i)
     & =  \psi_i( x_i + 1 \text{ mod } 2, d(v_i)-k_i + 2 ) \\
     & =  ( (x_i + 1) + 1 \text{ mod } 2, d(v_i)-(d(v_i)-k_i+2) + 2)
       =  (x_i, k_i) \;,
\end{align*}
from which it follows that $\psi$ is invertible and hence a bijection.
To establish~\eqref{eq:conj}, we rewrite that equation as
\begin{equation*}
 \prod_i F^\uparrow_{\pi(i)} =
     \prod_i \bigl( \psi^{-1} \circ \F^\downarrow_{\pi(i)} \circ \psi \bigr)\;.
\end{equation*}
Since
$\psi^{-1} = \psi$ it is sufficient to establish that
\begin{equation}
\label{eq:conji}
  \F^\uparrow_i = \psi \circ F^\downarrow_i \circ \psi
\end{equation}
for each $i \in \{1,2,\ldots n\}$. By the structure of the local maps
and $\psi$, the conjugation relation clearly holds in all coordinates
$j\ne~i$ when we evaluate the two sides of~\eqref{eq:conji} at a
state $s$. This leaves three cases to consider:

{\bfseries Case 1:} \emph{$x_i$ is mapped from~$0$ to~$1$ by
  $\F^\uparrow_i$}, which is possible only if $x_i = 0$,
$\sigma(x[i]) \geq k_i$, and $(\F^\uparrow_i(s))_i = (1, k_i+1)$. With
$s' = \psi( s ) = (x_j+1 \text{ mod } 2, d(v_j) - k_j + 2)_j$ we see
that $x'_i = 1$, $k'_i = d(i) - k_i + 2$, and $\sigma( x'[i] ) = d(i)
- \sigma( x[i] ) + 1 \le d(i) - k_i + 1$. Therefore, $\sigma( x'[i] )
\le k'_i$ and $\F^\downarrow_i$ maps $s'$ to $s''$ where $x''_i = 0$
and $k''_i = d(i) - k_i + 1$. Then $\psi(s'')_i = (1, d(i) -
(d(i)-k_i+1) + 2) = (1, k_i + 1)$, establishing Case~1.

{\bfseries Case 2:} \emph{$x_i$ is mapped from~$1$ to~$0$ by
  $\F^\uparrow_i$}. This case is similar to the first case. Here $x_i
= 1$, $\sigma(x[i]) \le k_i - 1$, and $(\F^\uparrow_i(s))_i = (0,
k_i)$.  Setting $s' = \psi( s ) = (x_j+1 \text{ mod } 2, d(v_j) -
k_j + 2)_j$, we see that $x'_i = 0$, $k'_i = d(i) - k_i + 2$, and
$\sigma( x'[i] ) = d(i) - \sigma( x[i] ) + 1 \ge d(i) - k_i + 2 =
k'_i$.  Since $\sigma( x'[i] ) \ge k'_i$ we see that $\F^\downarrow_i$
maps $s'$ to $s''$ where $x''_i = 1$ and $k''_i = k'_i = d(i) - k_i +
2$. Then $\psi(s'')_i = (0, d(i) - (d(i)-k_i+2) + 2) = (0, k_i)$,
establishing Case~2.

{\bfseries Case 3:} \emph{$x_i$ is mapped identically by
  $\F^\uparrow_i$}. In this case we either have ($i$) $x_i = 1$ and
$\sigma(x[i]) \ge k_i$ or ($ii$) $x_i = 0$ and $\sigma(x[i]) \le k_i -
1$.  In both cases we have $(\F^\uparrow_i(s))_i = (x_i, k_i)$ and $s'
= \psi(s) = ((x_j + 1 \text{ mod } 2, d(j) - k_j + 2))_j$. For case
($i$) $\sigma(x'[i]) = d(i) - \sigma(x[i]) + 1 \le k'_i$, so with $s''
= \F^\downarrow_i(s')$ we have $x''_i = 0$ and $k''_i = d(i) - k_i +
2$. It follows that for this case $\psi(s'')_i = (1, k_i)$ as
desired. Case ($ii$) is virtually identical with $\sigma(x'[i]) = d(i) -
\sigma(x[i]) + 1 \ge d(i) - (k_1-1) + 1 \ge k'_i$ leading to $x''_i =
1$ and $k''_i = d(i) - k_i + 2$ where $s'' = \F^\downarrow_i(s')$, and
we have $\psi(s'')_i = (0, k_i)$ as required, concluding the proof for
\etsds.
\end{proof}
The argument in the proof combined with the structure of the
map~$\psi$ ensure that~\eqref{eq:conj} holds with $\F^\uparrow_{\pi}$
and $\F^\downarrow_{\pi}$ replaced by $\F^\uparrow$ and
$\F^\downarrow$, respectively:
\begin{corollary}
\label{cor:conjugacy}
For any graph X, the maps $\F^\uparrow$ and $\F^\downarrow$ are topologically
conjugated.
\end{corollary}


To conclude the proof of Theorem~\ref{thm:main} we next turn to the
case of mixed dynamic threshold \sds. The following result extends an
earlier result in~\cite{Barrett:06a} for standard threshold systems.
\begin{proposition}
\label{prop:mixlimit}
  The \etsds map $\F^\updownarrow_\pi$ has no periodic orbits of
  length~$\ge 2$.
\end{proposition}
\begin{proof}
Since individual threshold components (i.e., the $k_i$'s) are not
necessarily monotone along orbits for $\F^\updownarrow_\pi$,
the previous arguments involving Lemma~\ref{lem:classic} cannot be put
to use directly. Instead we use an extension of a potential function argument
from~\cite{Barrett:06a}.

Let $X$ be a graph, $s = (x,k) \in \mathcal{S}$, $v\in\vset[X]$ and $e\in
\eset[X]$. Also, let $T_1(s,v)$ denote the (dynamic) threshold value for
vertex $v$ (i.e., $k_v$), and let $T_0(s,v)$ denote the smallest number of
states in $x[v]$ that must be zero to ensure that $x_v$ is mapped to
zero. Clearly, we have the relation $\bigl(d(v)+1\bigr) - T_0(s,v) = T_1(s,v)
- 1$, or $T_0(s,v) + T_1(s,v) = d(v) + 2$.  We next introduce the
vertex potentials
\begin{equation*}
P(s, v) =
\begin{cases}
T_1(s,v), & x_v = 1\\
T_0(s,v), & x_v = 0
\end{cases}
\end{equation*}
and the edge potentials
\begin{equation*}
P(s,e = \{u,v\}) =
\begin{cases}
1, & x_u \ne x_v \\
0, & \text{otherwise,}
\end{cases}
\end{equation*}
where $e = \{u,v\}\in\eset[X]$.
The potential function $P \colon \mathcal{S} \longrightarrow \N$ is
defined as the sum of all vertex and edge potentials:
\begin{equation}
  \label{eq:potential}
  P(s) = \sum_{v \in\vset[X]} P(s,v) + \sum_{e \in \eset[X]} P(s,e)
\end{equation}
Clearly, there exists positive integers $m$ and $M$ such that $m \le P(s) \le
M$ for all $s\in \mathcal{S}$. Set $s' = \F^\updownarrow_i(s)$.

{\bfseries Claim:} \emph{for each $s\in \mathcal{S}$ such that
  $\F^\updownarrow_v$ maps $x_v$ non-identically we have $P(s') <
  P(s)$}. Clearly, any change in potential is solely associated with
vertex $v$ and edges incident with $v$. We write $n_0(s,v)$ and
$n_1(s,v)$) for the number of neighbors of $v$ with $x_v = 0$ and $x_v
= 1$, respectively. There are two possible cases:

{\bfseries Case 1:} If \emph{$x_v$ is mapped from $0$ to $1$} then we have
$n_1(s,v) \ge T_1(s,v)$, and, using the identity $T_0(v) + T_1(v)
= d(v) + 2$ we get $n_0(s,v) \le T_0(s,v) - 2$. Since $T_1(s', v) =
T_1(s,v) + 1$ we obtain
\begin{align*}
P(s') - P(s)
  &= T_1(s', v) + n_0(s', v) - [ T_0(s,v) + n_1(s,v)] \\
  &\le [T_1(s,v) + 1] + T_0(s,v) - 2 - T_0(s,v) - n_1(s,v) \le -1 \;.
\end{align*}

{\bfseries Case 2:} When \emph{$x_v$ is mapped from $1$ to $0$} we have
$n_0(s,v) \ge T_0(s,v)$. This gives $n_1(s,v) \le T_1(s,v) - 2$. We
also have $T_0(s,v) + 1 = T_0(s',v)$ which yields
\begin{align*}
P(s') - P(s)
  &= T_0(s',v) + n_1(s',v) - [ T_1(s, v) + n_0(s, v)] \\
  &= T_0(s,v) + 1 + n_1(s,v) - T_1(s,v) - n_0(s,v) \\
  &\le n_0(s,v) + 1 + T_1(s,v) - 2 - T_1(s,v) - n_0(s,v) = -1 \;,
\end{align*}
proving the claim. In light of the compositional structure of
$\F^\updownarrow_\pi$ and the boundedness of the potential function
$P$, it follows immediately that $\F^\updownarrow_\pi$ cannot have
periodic orbits of length $\ge 2$ since that would cause an immediate
contradiction.
This completes the proof of Proposition~\ref{prop:mixlimit}.
\end{proof}
\begin{remark}
The particular choice of potential function used in the previous proof
does not work directly for the synchronous case. To see this, consider
the circle graph on four vertices with state $s = ( (1,2), (0,2),
(1,2), (0,2) )$, which maps to $s' = ( (0,1), (1,3), (0,1), (1,3))$,
which in turn maps to $s$. Clearly, the edge potential is $4$ for both
states whereas the vertex potential of $s$ and $s'$ is $8$ and $12$,
respectively.
\end{remark}


The proof of Theorem~\ref{thm:main} is completed by the following:
\begin{proposition}
\label{prop:mixsynch}
The \etgca map $\F^\updownarrow$ has no periodic orbits of length $\ge 3$.
\end{proposition}
Clearly, $\F^\updownarrow$ may have fixed points. By the previous
remark, we see that $2$-cycles can occur for this class of maps.
\begin{proof}
This proof builds on a construction originally given by Goles \& Olivos
in~\cite{Goles:80} for neural networks and an extension of their proof that we
developed in~\cite{Kuhlman:12} that was needed to analyze the class of
\emph{bi-threshold} systems.

As in~\cite{Goles:80}, we define the extended threshold function
(neural networks)
\begin{equation}
\label{eq:genthreshold2}
\tau'_i(x_1,\dots,x_n) =
\begin{cases}
   1   & \text{ if } x_i=0 \text{ and }\sum\limits^n_{j=1}a_{ij}x_j \ge k  \\
   0   & \text{ if } x_i=1 \text{ and }\sum\limits^n_{j=1}a_{ij}x_j < k  \\
  x_i  &   \text{ otherwise,}
\end{cases}
\end{equation}
where $A = (a_{ij})_{ij}$ is a symmetric, real-valued matrix of
dimension $n\times n$. Here $n = \card{\vset[X]}$. Our result, which
we prove for $\tau'$, will follow by specializing to
the case where $a_{ij} = 1$ if $\{i,j\}\in\eset[X]$ and $a_{ij} = 0$
otherwise.

To start, let $s = (x,k) \in \mathcal{S}$ and assume that $s$ reaches
a periodic orbit of size $T$ under $\F^\updownarrow$. We let $\{z(0),
z(1), \ldots, z(T-1)\}$ denote the sequence of (extended) states on
this orbit, write $z_i$ for the projection of this sequence onto its
$i^{\text{th}}$ component, and let $\gamma_i$ be the period of
$z_i$. Clearly, $\gamma_i$ must divide $T$. Let $Z = \{z_1, z_2,
\ldots, z_n\}$ and define the function $L \colon Z \times
Z\rightarrow\mathbb{R}$ by
\begin{equation*}
  L(z_i, z_j) = a_{ij} \sum^{T-1}_{l=0} (x_j(l+1) - x_j(l-1)) x_i(l)\;,
\end{equation*}
with indices taken modulo~$T$. Note again that the $z_j$s are extended
system states. The operator $L$ has the following
properties (see~\cite{Goles:80,Kuhlman:12}):
\begin{itemize}
\item[($i$)] $L(z_i,z_j)+L(z_j,z_i) = 0$ for $i,j \in \{1,\ldots,n\}$
  (anti-symmetry).
\item[($ii$)] If $\gamma_i\le 2$ then $L(z_i,z_j)=0$ for $j \in
  \{1,\ldots,n\}$.
\end{itemize}
Let $z_i\in Z$ and suppose in the following that $\gamma_i \ge 3$. We
set
\begin{equation*}
  \supp(z_i) = \{ l\in\{0,\ldots,T-1\} : x_i(l) = 1\}\;,
\end{equation*}
and write $\mathcal{I}(l) = \{l, l+2, l+4, \ldots, l-4, l-2\}$. Next, set
\begin{equation*}
C_0 =
\begin{cases}
\varnothing,& \text{ if there is no }l_0\in\{0,\dots,T-1\} \text{ such
  that } \mathcal{I}(l_0) \subset \supp(z_i)\\
  \mathcal{I}(l_0),& \text{otherwise}.
\end{cases}
\end{equation*}
We define $C_1$ as the set
\begin{equation*}
  C_1 = \{l_1+2b \in \supp(z_i) : b = 0,1,\ldots, q_1\} \;,
\end{equation*}
where $l_1$ is the smallest index not in $C_0$ satisfying $x_i(l_1-2)
= 0$ and $q_1$ satisfies $x_i(l_1+2q_1+2) = 0$.
For $k \ge 2$ we define the sets $C_k$ by
\begin{equation*}
C_k = \{ l_k + 2b \in \supp(z_i) : b = 0,1,\dots,q_k\}\;,
\end{equation*}
where $l_k = l_{k-1} + r_k \pmod T \notin \{l_1,\ldots,l_{k-1}\}$ is
the smallest index for which $x_i(l_k-2) = 0$ and $q_k$ satisfies
$x_i(l_k+2q_k+2) = 0$.

Since $\gamma_i \ge 3$ by assumption, there always exists $l_1\in
\supp(z_i)$ for which $z_i(l_1-2) = 0$. This allows us to construct
$\mathcal{C} = \{C_0,\dots,C_p\}$ which is a partition of $\supp(z_i)$.

We will show that if $\gamma_i \ge 3$, then we are led to the conclusion that
\begin{equation*}
  \sum^n_{j=1}L(z_i,z_j)<0 \;.
\end{equation*}
As in~\cite{Kuhlman:12}, we can rewrite
\begin{align*}
\sum^n_{j=1} L(z_i, z_j) &= \sum^p_{k=0} \Psi_{ik}\;,
\end{align*}
where
\begin{equation}
\label{eq:Psi}
\Psi_{ik} = \sum^n_{j=1} a_{ij} \sum_{l\in C_k} (x_j(l+1)-x_j(l-1)) \;.
\end{equation}
If $C_0 = \emptyset$ then $\Psi_{i0} = 0$, and if $C_0 = \{l_0,l_0+2,
\ldots, l_0-2\}$ we have
\begin{equation*}
\sum_{l\in C_0} (z_j(l+1)-z_j(l-1)) = 0\;.
\end{equation*}
In other words, we always have $\Psi_{i0} = 0$. Assume $k>0$ in
the following.
From the assumption that $\gamma_i \ge 3$, there exists $C_k\ne
\varnothing$ such that $C_k = \{l_k,l_k+2,\ldots,l_k+2q_k\}$, so we
can re-write $\Psi_{ik}$ as
\begin{align*}
\Psi_{ik} &= \sum^n_{j=1}a_{ij}\sum_{s=0}^{q_k} (x_j(l_k+2s+1)-x_j(l_k+2s-1))\\
         &= \sum^n_{j=1}a_{ij} x_j(l_k+2q_k+1) - \sum^n_{j=1} a_{ij} x_j(l_k-1)\;.
\end{align*}

The dynamic threshold functions require that we distinguish the elements
of~$\mathcal{C}$ more carefully than what was needed in~\cite{Goles:80}. An
element $C\in \mathcal{C}$ is of \emph{type} $ab$ if $C = (l, l+2, l+4,\ldots,
k)$ and $x_{l-1} = a$ and $x_{k+1} = b$ where all indices are taken modulo
$T$. Here we write $m_{ab} = m_{ab}(\mathcal{C})$ for the number of elements
of $\mathcal{C}$ of type~$ab$. A key property needed here is that $m_{01} =
m_{10}$. A proof of this fact is given in~\cite{Kuhlman:12}.

In the following four cases we assume that the threshold of vertex $i$
at time $l_k-2$ is $k$.

\medskip
\noindent \emph{{\bfseries $C_k$ is of type $00$}}: in this case
$x_i(l_k-1) = 0$, $x_i(l_k) = 1$, $x_i(l_k+2q_k+1) = 0$ and
$x_i(l_k+2q_k+2) = 0$, which is only possible if
\begin{equation*}
\sum_{j=1}^n a_{ij} z_j(l_k-1) \ge k
\quad \text{and} \quad
\sum_{j=1}^n a_{ij} z_j(l_k+2q_k+1) < k \;,
\end{equation*}
which implies that $\Psi_{ik} < 0$.

\medskip
\noindent \emph{{\bfseries $C_k$ is of type $11$}}: this case is
completely analogous to the $00$ case, and
\begin{equation*}
\sum_{j=1}^n a_{ij} x_j(l_k-1) \ge k+1
\quad \text{and} \quad
\sum_{j=1}^n a_{ij} x_j(l_k+2q_k+1) < k+1 \;,
\end{equation*}
leading to $\Psi_{ik} < 0$.

\medskip
\noindent \emph{{\bfseries $C_k$ is of type $10$}}: here
$x_i(l_k-1) = 1$, $x_i(l_k) = 1$, $x_i(l_k+2q_k+1) = 0$ and
$x_i(l_k+2q_k+2) = 0$. This implies that
\begin{equation*}
\sum_{j=1} a_{ij}x_j(l_k-1) \ge k+1
\quad\text{and}\quad
\sum_{j=1} a_{ij}x_j(l_k+2q_k+1) < k \;,
\end{equation*}
so that $\Psi_{ik} < k - (k+1) = -1$.

\medskip
\noindent \emph{{\bfseries $C_k$ is of type $01$}}: this case is
essentially the same as the $10$ case, but here $\Psi_{ik} < k+1 - k =
1$.

\medskip

Using the above four cases, we now have
\begin{equation*}
\sum_{j=0}^n L(z_i, z_j) = \sum_{k=0}^p \Psi_{ik}
< 0
+ m_{00} \cdot 0
+ m_{11} \cdot 0
+ m_{10} (-1)
+ m_{01} (+1) = 0\;,
\end{equation*}
where the last equality follows from the fact that $m_{10} =
m_{01}$. As in the original proof, we see that the assumption
$\gamma_i\ge 3$ leads to a contradiction since the sum of all terms
$L(z_i, z_j)$ is zero by anti-symmetry. In effect, all the component
periods $\gamma_i$ for $1\le i \le n$ are bounded above by $2$ which
leads to the desired conclusion that $T \le 2$. This also concludes
the current proof as well as the proof of Theorem~\ref{thm:main}.
\end{proof}

Increasing, decreasing and mixed \etsds also have a certain
monotonicity property along orbits.  We say that a state transition
$(x,k) \mapsto (x',k')$ is \emph{unidirectional} if all non-trivial
transitions $x_v \mapsto x_v'$ with $x_v \ne x_v'$ are either all of
the form (a) $0\mapsto 1$ or (b) all of the form $1 \mapsto 0$. The
following result is another extension of a result
in~\cite{Barrett:06a} for standard threshold \sds maps to \etsds maps.
\begin{proposition}
Let $\phi \in \{\F^\updownarrow_\pi, \F^\uparrow_\pi, \F^\downarrow_\pi\}$ and
assume that for the state $s\in \mathcal{S}$ the transition $s \mapsto
\phi(s)$ is unidirectional. Then all transitions along the forward orbit
originating at $s$ are also unidirectional and in the same direction as the
one of $s \mapsto \phi(s)$.
\end{proposition}
\begin{proof}
Let $\phi = \F^\updownarrow_\pi$, let $s, s', s''\in \mathcal{S}$ and
assume that $s \mapsto s'$ unidirectionally with all transitions going
from~$0$ to~$1$. Assume next that $s' \mapsto s''$ but not
unidirectionally, and let~$i$ denote the vertex minimal with respect
to $\pi$ for which~$1$ maps to~$0$ in the second transition. Since the
previous transition was unidirectional ($0\mapsto 1$), we must have
$n_1(s,i) < n_1(s',i)$ at the time of the vertex update for $i$ in the
second transition. This yields a contradiction regardless of whether
the threshold $k_i$ was mapped to $k_i+1$ or not in the first
transition. The case where $s\mapsto s'$ unidirectionally with all
transitions going from $1$ to $0$ is analogous.
\end{proof}


\section{Structure and Enumeration of Fixed Point}
\label{sec:enumeration}

This section is concerned with enumeration and characterization of
fixed points of \etgds maps. Here we note that for a fixed sequence of
vertex functions over a given graph $X$, the set of fixed points is
the same whether we use a parallel or a permutation sequential
update~\cite{Mortveit:07}. That is a general and well-known fact. We
cover three graph classes: trees, the \emph{path graph} $\Path_n$ on
$n$ vertices (precursor for result on trees in next section), and the
\emph{circle graph} on $n$ vertices denoted by $\Circ_n$. However we
first have the following result:
\begin{proposition}
\label{prop:fix}
For any graph $X$ and any $\pi\in S_X$ we have
\begin{equation}
\label{eq:fixset}
 \Fix( \F^\uparrow ) =
 \Fix( \F^\downarrow ) =
 \Fix( \F^\updownarrow ) =
 \Fix( \F^\uparrow_\pi ) =
 \Fix( \F^\downarrow_\pi ) =
 \Fix( \F^\updownarrow_\pi ) \;.
\end{equation}
\end{proposition}
\begin{proof}
For a graph $X$, a sub-configuration $x[v]$ is a \emph{local fixed
  point} if $f_v(x[v]) = x_v$. Two local fixed points $\xi_i$ and
$\xi_j$ are \emph{compatible} whenever they agree on the intersection
$n[i] \cap n[j]$. Clearly, there is a bijective correspondence between
the set of fixed points and the set of local fixed point sequences
$\xi = (\xi_1, \ldots, \xi_n)$ whose components are pairwise
compatible. It is straightforward to see that for $f^\uparrow_i$,
$f^\downarrow_i$ and $f^\updownarrow_i$ to have $\xi_i$ as a local
fixed point, necessary and sufficient conditions are in all cases
\begin{equation}
\label{eq:fix_cond}
\bigl( x_i = 0 \text{ and } \sigma( x[i] ) < k_i \bigr)
\quad\text{or}\quad
\bigl( x_i = 1 \text{ and } \sigma( x[i] ) \ge k_i \bigr) \;,
\end{equation}
and the proof follows.
\end{proof}
Of course, the transient dynamics for these system classes will
generally differ. The following result further characterizes the set
$\Fix( \F^\uparrow )$ appearing in~\eqref{eq:fixset}.
\begin{lemma}
\label{lem:state1}
For any graph $X$ and any vertex $v\in\vset[X]$ we have $x_v = 1$ for
precisely half of the elements of the set $\Fix( \F^\uparrow )$.
\end{lemma}
\begin{proof}
The map $\psi$ in~\eqref{eq:psiconj} induces a bijection $\psi' \colon
\Fix( \F^\uparrow ) \longrightarrow \Fix( \F^\downarrow )$ by
restriction. The result now follows from that that ($i$) $\Fix(
\F^\uparrow ) = \Fix( \F^\downarrow )$ and ($ii$) the image of any
fixed point $s$ for which $x_v = 1$ is a fixed point with $x_v = 0$
and vice versa.
\end{proof}
A practical consequence of this results is that it can simplify fixed
point enumeration.


\subsection{Fixed Points of \etgds over $\Path_n$.}
\label{sec:path}

In this section we enumerate the fixed points of dynamic threshold \gds over
$\Path_n$. Both the result and its proof are relevant for the argument
covering the case of trees in Section~\ref{sec:trees}. Let $\Fib{n}$ denote
the $n^{\text{th}}$ Fibonacci number ($\Fib{0} = 0$, $\Fib{1} = 1$ and
$\Fib{n} = \Fib{n-1} + \Fib{n-2}$ for $n\ge 2$). We will also write $\fib(n)$
for $\Fib{n}$.
\begin{proposition}
\label{prop:npath}
For $X = \Path_n$ we have $\card{\Fix(\F^\uparrow)} = 2\Fib{3n-1}$\;.
\end{proposition}
\begin{proof}
We proceed in a recursive manner constructing the fixed points over
$\Path_{n+1}$ from those over $\Path_n$. There are two disjoint sets
of fixed points to consider: ($i$) the fixed points over $\Path_{n+1}$
whose restrictions are fixed points over $\Path_n$, and ($ii$) their
complement. Here we derive and solve a recursion relation where the
second class of fixed points are charged to the first class. In this
way, we can directly relate fixed points in the~$(n+1)$~case to
the~$n$~case and~$(n-1)$-case.

Clearly, the set of fixed points $s = (s_1, \ldots, s_n)$ over
$\Path_n$ fall into six disjoint sets depending on the possible values
for $(s_{n-1}, s_n)$ which are
$\bigl( (0,*), (0,1) \bigr)$,
$\bigl( (0,*), (0,2) \bigr)$,
$\bigl( (0,*), (1,1) \bigr)$,
$\bigl( (1,*), (0,2) \bigr)$,
$\bigl( (1,*), (1,1) \bigr)$, and
$\bigl( (1,*), (1,2) \bigr)$
where $*$ denotes any element of $\{1,2,3\}$. Figure~\ref{fig:5} shows
the extended and charged fixed points in each of these six cases.

\begin{figure}[ht]
\centerline{\includegraphics[scale=1]{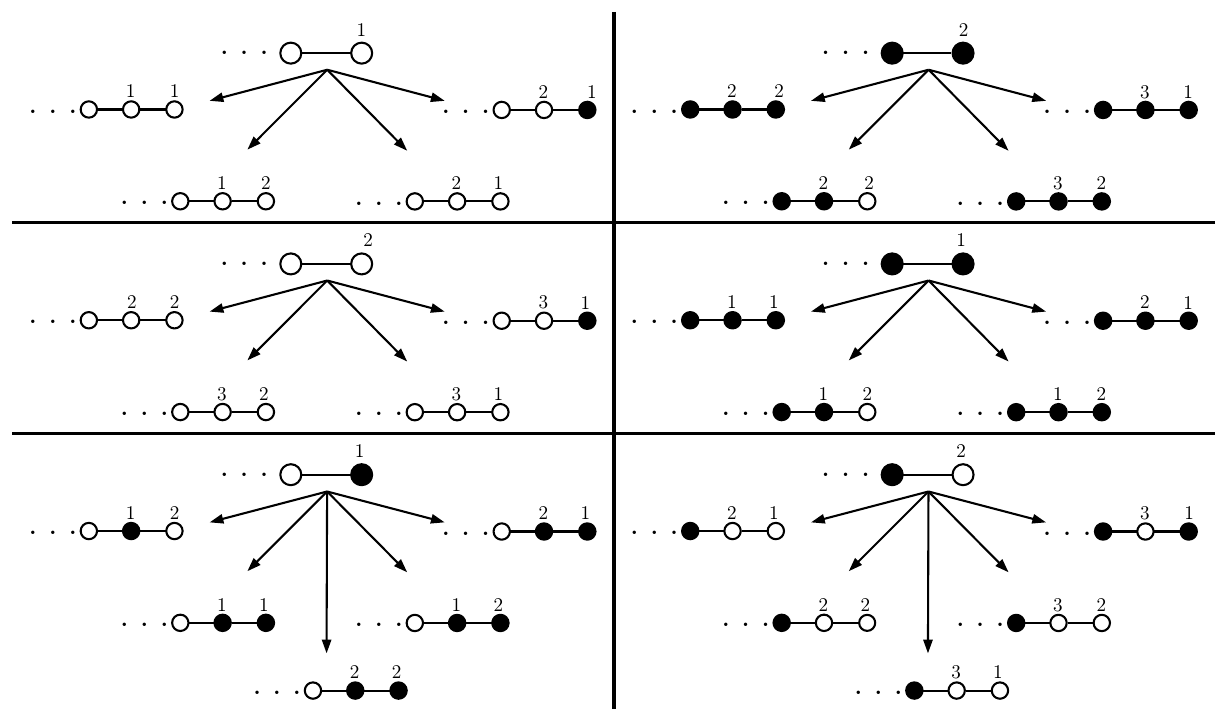}}
\caption{The extension \& charging scheme used in the proof of
  Proposition~\ref{prop:npath}. In the diagram, a filled (resp. empty)
  circle encodes a state that is~$1$ (resp.~$0$) while numbers give
  the threshold value.}
\label{fig:5}
\end{figure}
We set $\Fix(n) = \card{\Fix(\F^\uparrow)}$. Here, the two cases displayed on
the last row in Figure~\ref{fig:5} give rise to five fixed points. Each
fixed point over $\Path_{n-1}$ can be adjoined to precisely one of these cases, so
the number of fixed points over $\Path_{n+1}$ accounted for by these is $5
\Fix(n-1)$. Each of the remaining cases account for four fixed points and
there are $\Fix(n)-\Fix(n-1)$ of these.

Direct computations give $\Fix(2) = 10 = 2\Fib{5}$, $\Fix(3) = 42 = 2
\Fib{8}$ and $\Fix(4) = 2 \Fib{11}$. We proceed by induction and
assume that the $\Fix(n) = 2\Fib{3n-1}$ holds for $n\ge 4$. Then
\begin{eqnarray*}
\Fix(n+1)
  & = & 5\Fix(n-1) + 4\left[\Fix(n)-\Fix(n-1)\right] \\
  & = & 5\cdot 2\Fib{3n-4} + 4 \left[2\Fib{3n-1}-2\Fib{3n-4}\right] \\
  & = & 2\Fib{3(n+1)-1}\;,
\end{eqnarray*}
and we are done.
\end{proof}

{\bfseries Scaling:} It is known that $\Fib{n}$ can be determined as the
integer closest to $\varphi^n/\sqrt{5}$ where $\varphi$ is the golden
ratio. Thus, we see that the number of fixed points is roughly given by
\begin{equation*}
\label{eq:path}
\bigl( 1 - {1}/\sqrt{5} \bigr) \bigl(\varphi^3\bigr)^n
 = \bigl( 1 - {1}/\sqrt{5}\bigr) \bigl(2+\sqrt{5}\bigr)^n \;.
\end{equation*}
For comparison, we note that for $\Path_n$ there are $(4/9)\cdot
6^n$ states in phase space.


\subsection{Fixed Points of \etgds over $\Circ_n$.}
\label{sec:circ}

Specializing the proof of Proposition~\ref{prop:fix}, we here derive a
recursion relation for the number of fixed points for \etsds/\etgca
over the graph $\Circ_n$. Let $\Luc_n$ denote the $n^{\text{th}}$ Lucas
number ($\Luc_0 = 2$, $\Luc_1 = 1$ and $\Luc_n = \Luc_{n-1} +
\Luc_{n-2}$ for $n\ge 2$).
\begin{proposition}
\label{prop:circ}
For $X = \Circ_n$ we have $\card{\Fix(\F^\uparrow)} = 2 + \Luc_{3n-1}$.
\end{proposition}
\begin{proof}
The approach, which at its core is based on the matrix transfer
method~\cite{Stanley:00}, patches the local fixed points to construct
global fixed points, see~\cite[Chap. 5]{Mortveit:07}. In the case of
$\Circ_n$, local fixed points are of the form
\begin{equation*}
\xi_i = \bigl( (x_{i-1}, x_i, x_{i+1}), (k_{i-1}, k_i, k_{i+1}) \bigr) \;,
\end{equation*}
with indices taken modulo $n$. Such a tuple $\xi_i$ is a local fixed
point for $\F^\uparrow_\pi$ if $f^\uparrow_i( \xi_i ) = (x_i,
k_i)$. Moreover, two local fixed points $\xi_i$ and $\xi_j$ are
\emph{compatible} whenever they agree on their intersection $n[i] \cap
n[j]$ which is trivial unless $i$ and $j$ are $\le 2$ apart modulo
$n$. We write $\xi_i \triangleleft \xi_{i+1}$ for compatible and
consecutive local fixed points. Clearly, there is a bijective
correspondence between the set of sequences of the form $\xi = (\xi_1,
\ldots, \xi_n)$ satisfying the conditions
\begin{equation*}
\xi_1 \triangleleft \xi_2 \triangleleft \cdots \triangleleft \xi_n
\triangleleft \xi_1
\end{equation*}
and the set of fixed points of $\F^\uparrow_\pi$ over $\Circ_n$. In
the case of $\Circ_n$ there are $6$ possible vertex states and $6^3 =
216$ tuples $\xi_i$, $144$ of which are local fixed points. Extending
the approach in~\cite[Chap. 5]{Mortveit:07}, we construct the
adjacency matrix $A$ for the graph $G$ whose vertices are the local
fixed points and with directed edges all  $(\xi_i, \xi_{i'})$ such
that $\xi_i \triangleleft \xi_{i'}$
Using symbolic algebra software, it is straightforward to
verify that the characteristic polynomial of $A$ is
\begin{equation*}
  \chi_A(x) = x^{144} - 6x^{143} + 8x^{142} - 2x^{141} - x^{140} \;.
\end{equation*}
It follows that the number of fixed points $L_n$ over $\Circ_n$ for
$\F^\uparrow_\pi$ satisfies the recursion relation
\begin{equation*}
  L_n = 6L_{n-1} + 8 L_{n-2} - 2L_{n-3} - L_{n-4}
\end{equation*}
with initial values $L_3 = 78$, $L_4 = 324$, $L_5 = 1366$ and $L_6 =
5780$ obtained through direct calculations. This may be solved to give
the explicit formula
\begin{equation}
\label{eq:circ}
L_n = 2 + (2+\sqrt{5})^n + (2-\sqrt{5})^n
    = 2 + \varphi^{3n} + (1-\varphi)^{3n}
    = 2 + \Luc_{3n}\;,
\end{equation}
completing the proof.
\end{proof}
The simple form of~\eqref{eq:circ} indicates that it may be possible to
construct a proof similar to the one for the graph $\Path_n$ in the previous
section but instead using suitable extensions of fixed points over
$\Circ_{n-1}$ and $\Circ_{n-2}$.

Clearly, the number of fixed points scales with $n$ as $(2+\sqrt{5})^n$
whereas the number of total states in phase space is~$6^n$. We thus see that
the number of fixed points over $\Path_n$ and $\Circ_n$ essentially only
differ by the constant factor $1-\tfrac{1}{\sqrt{5}}$.

We remark that the method employed here for $\etgds$ and $\Circ_n$ can
be extended to the graph $\Circ_{n,r}$ where each vertex $i$ is
connected to all vertices $j$ for which the graph distance in
$\Circ_n$ is less than or equal to~$r$; this corresponds to the setting of
radius-$r$ elementary cellular automata. However, the book-keeping
involved makes it somewhat impractical.


\subsection{Fixed Points over Trees.}
\label{sec:trees}

In this section we describe an algorithm that can be used to determine
the number of fixed points in~\eqref{eq:fixset} in the case where the
graph $X$ is a tree. The generalizations of this algorithm to general
GDS maps over trees is being pursued elsewhere, so we only outline the
ideas and illustrate with an example.  

To avoid any confusion we will write $T$ instead of $X$ for the graph
to emphasize this point when needed.  The algorithm considers a tree
as the \emph{union of paths}, such that any two paths have at most one
vertex in common.  Recall that the \emph{union} of two graphs $X_1$
and $X_2$ is the graph $X = X_1 \cup X_2$ with $\vset[X]
= \vset[X_1] \cup \vset[X_2]$ and $\eset[X] =
\eset[X_1] \cup \eset[X_2]$. Clearly, a tree $T$ can be described as a union
of paths $p_i$ for which any pair of paths have at most one vertex in common.
We denote the number of fixed points by $\Fix(T)$, and for $v \in
\vset[X]$ we write $\chi(v; X)$ for the number of fixed points
over~$X$ with $i^{\text{th}}$ component $s_i = (0,d_X(i)+1)$. Here we
use $X$ as subscript in the degree $d_X(i)$ since $i$ will appear as a vertex
in multiple graphs.

\begin{figure}[ht]
\centerline{\includegraphics[width=0.9\linewidth]{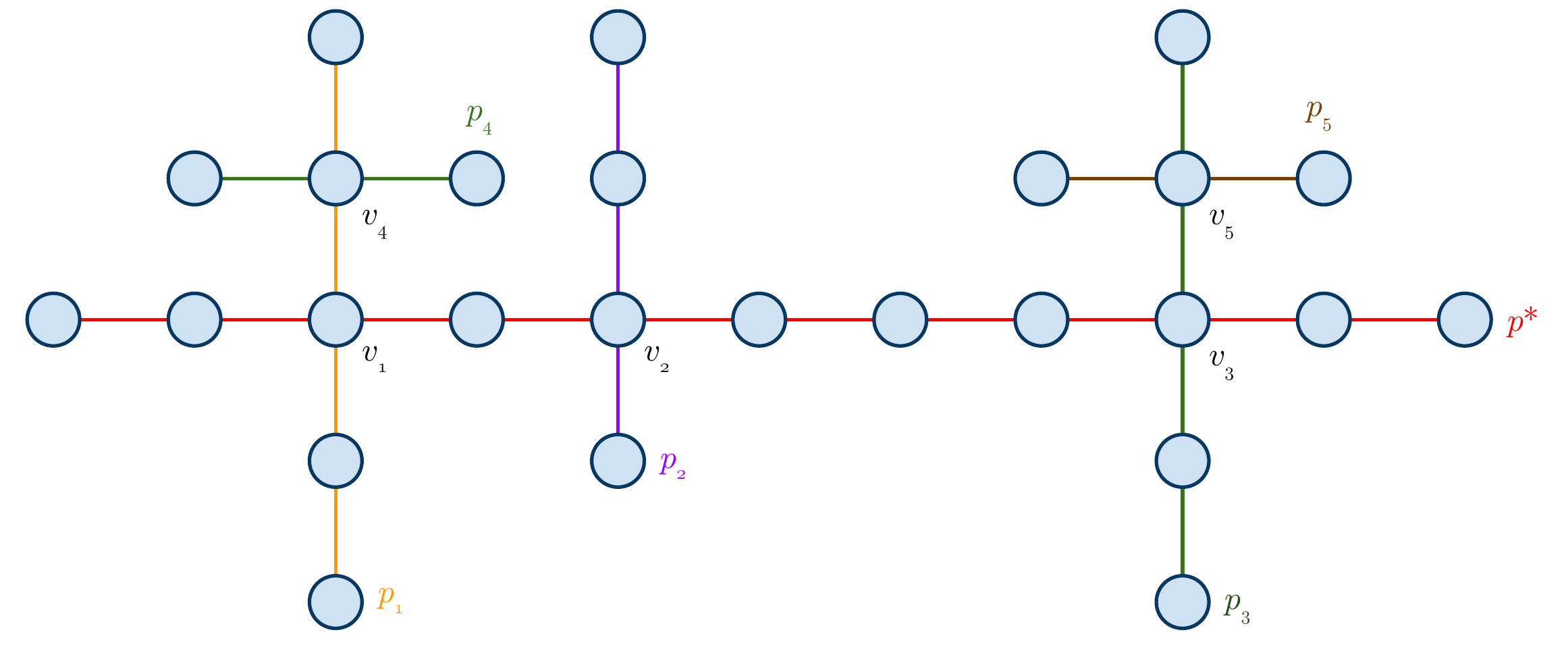}}
\caption{A tree illustrating the algorithm to determine $\Fix(T)$ of
  Section~\ref{sec:trees}. }
\label{fig:tree}
\end{figure}

As a specific example to illustrate the algorithm, we will consider
the tree~$T$ shown in Figure~\ref{fig:tree}.  This tree has been split
into paths $p^*$ and $p_1$ through $p_5$.  The general idea is to
first consider the graph consisting only of path $p^*$, and then find
the fixed points as each additional path is appended in the order
$p_1$, $p_2$, $p_3$, $p_4$, and $p_5$.  We denote by $T_1$ the tree
consisting of $p^*$ and $p_1$, by $T_2$ the tree $T_1$ with $p_2$ and
so on. That is, we set $T_0 = p^*$ and $T_i = T_{i-1} \cup p_i$ for
$1 \le i \le 5$. To compute the fixed points of $T_i$, we have to keep
track of the vertices $v_i$ where paths intersect.

The initial path $p^*$ contains the three intersection vertices $v_1$,
$v_2$, and $v_3$.  For each such vertex $v_i$, we need the number of
fixed points of $p^*$ where $s_i = (0,d_{p^*}(i)+1)$ which is given
by the following proposition.
\begin{proposition}
\label{prop:a2}
Let $p$ be a path on $n$ vertices and denote by $r_i = r_i(p)$ the minimal
distance (standard graph metric) of the vertex $v_i\in p$ to and end of
$p$. Then we have
\begin{equation}
\label{eq:chi}
  \chi(p;v_i) = \fib \bigl(3(r_i+1)-2\bigr) \fib\bigl(3(n-r_i)-2\bigr) \;.
\end{equation}
\end{proposition}
\let\zeta\chi
To keep track of the vertices~$v_2$ and $v_3$ in trees $T_i$, we also need the
number of fixed points of $p^*$ where multiple vertices of intersection
simultaneously satisfy $s_i = (0,d_{p^*}(i)+1)$.  For a set
$\{v_{\alpha}\}$ of such vertices, we denote this quantity by
$\zeta(\{v_{\alpha}\}; p^*)$. Here we have:
\begin{proposition}
\label{prop:a3}
Let $p = (V=\{v_1,\dots,v_n\},\{e_1,\dots,e_{n-1}\})$ be a path, let
$\{v_{\alpha_1},\dots,v_{\alpha_m} \} \subset V$ with $\alpha_i<
\alpha_{i+1}$, and split $p$ into paths $p_1, \ldots, p_{m+1}$. Then the
number of fixed points over $p$ where $s_{\alpha_i} = (0,d_p(\alpha_i)+1)$ for all
$1 \le i \le m$ is given by
\begin{equation}
\label{eq:chiext}
\qquad \zeta(\{v_{\alpha}\}; p ) =
   \chi(v_1; p_1) \chi(v_1; p_{m+1})  \displaystyle\prod_{i=1}^{m-1}\,
  \displaystyle\sum_{j=1}^{\alpha_{i+1}-\alpha_i} \fib(3j-2) \;.
\end{equation}
\end{proposition}
With Propositions~\ref{prop:a2} and~\ref{prop:a3} in place we can now state
the final result needed:
\begin{theorem}
\label{thm:a1}
If $X = X_1 \cup X_2$, where $X_1$ and $X_2$ intersect in precisely
one vertex~$v$, then
\begin{equation}
\label{eq:treenfix}
  \Fix(X) = \chi(v; X_1) \Fix(X_2)
   + \bigl(\Fix(X_1)-2\cdot\chi(v; X_1)\bigr) \chi(v; X_2)\;.
\end{equation}
\end{theorem}
Using Theorem~\ref{thm:a1}, we first compute $\Fix(T_1)$ and proceed
iteratively until we have arrived at the full tree $T$. For the first step we obtain
\begin{align*}
\Fix(T_1) &= \chi(v_1; p^*)\Fix(p_1) +
       \bigl(\Fix(p^*)-2\chi(v_1; p^*)\bigr) \chi(v_1; p_1) \;, \\
 \chi(v_i, T_1) &= \zeta(\{v_1,v_i\}, p^*)\Fix(p_1) \\ &\phantom{=}+ 
     [\chi(v_i, p^*)-2\,\zeta(\{v_1,v_i\}, p^*)] \,\chi(v_1, p_1),\; i=2,3 \\
 \zeta(\{v_2,v_3\},T_1) &= \zeta(\{v_1,v_2,v_3\}, T_1) \Fix(p_1) \\
 &\phantom{=}+ [\zeta(\{v_2,v_3\}, p^*)-2\,\zeta(\{v_1,v_2,v_3\}, p^*)] \,\chi(v_1,p_1) \\
 \chi(v_4,T_1) &= \chi(v_1,p^*)\chi(v_1,p_1) + \left(\Fix(p_1) - 
  2\chi(v_1,p^*)\right)\zeta(\{v_1,v_4\}, p_1) \\
 \zeta(\{v_i,v_4\}, T_1) &= \zeta(\{v_1,v_i\}, p^*) \chi(v_4, p_1) \\
&\phantom{=}+ [\chi(v_i, p^*)-2\,\zeta(\{v_1,v_i\},p^*)] \zeta(\{v_1,v_4\}, p_1),\quad  i=2,3 \\
 \zeta(\{v_2,v_3,v_4\}, T_1) &= \zeta(\{v_1,v_2,v_3\},p^*)\chi(v_4,p_1) \\
 &  \phantom{=}+ [\zeta(\{v_2,v_3\},p^*)-2\,\zeta(\{v_1,v_2,v_3\},p^*)]\zeta(\{v_1,v_4\},p_1)
\end{align*}
where each of the terms involving $\chi$ are computed using~\eqref{eq:chi}
and~\eqref{eq:chiext} where the factor $\chi(v_2;T_1)$ can be determined via
Theorem~\ref{thm:a1}, Lemma~\ref{lem:state1} and
Proposition~\ref{prop:a3}. In the same manner we get
\begin{equation*}
  \Fix(T_i) = \chi(v_i; T_{i-1})\Fix(p_i) +
    \bigl(\Fix(T_{i-1}) - 2\chi(v_i; T_{i-1})\bigr)\chi(v_i; p_i) \;,
\end{equation*}
for $2 \le i \le 5$. Careful evaluation shows that 
\begin{align*}
\Fix(T_1) &= 1,142,003,642, \\
\Fix(T_2) &= 70,046,004,938, \\
\Fix(T_3) &= 18,361,190,404,154, \\
\chi(v_5, T_3) &= 4,096,066,198,731, \\
\Fix(T_4) &= 263,558,770,077,330, \\
\chi(v_5, T_4) &= 58,795,434,594,819, \text{and finally} \\
\Fix(T) &= \Fix(T_5) = 3,783,119,360,971,626 \;. 
\end{align*}
Again, the details and generalizations of this algorithm with
supporting proofs are being pursued elsewhere.


\section{Summary and Future Work}
\label{sec:summary}

In this paper we introduced \emph{dynamic threshold} graph dynamical
systems.  This setting provides a more realistic model of phenomena
involving complex contagions where the time evolution may affect the
threshold values of the system entities. The results we have given on
limit set structures and their enumeration may serve as a useful
starting point for the development of mathematical models for such
phenomena in a broad class of disciplines.

For application purposes, it is also desirable to know stability properties of
the limit sets. Is a fixed point stable under a class of state perturbations?
An example of such work is given in~\cite{Kumar:09}. Shedding light on the
stability of \etgds seems like a natural avenue for future work as does
generalizations of the threshold models considered here. The potential
function may be useful in this regard as it offers insight into the transient
structure of the system. In particular, this could give insight and bounds for
convergence rates for dynamic threshold systems.
From a combinatorial perspective, a more direct derivation of the
recurrence relation enumerating fixed points over the circle graph
through an extension/charging scheme would be desirable and may also
shed some more light on the structure of the fixed points.

Finally, the assumption that the transitions where $x_v$ switches from
$0$ to $1$ and from $1$ to $0$ are governed by the common threshold
$k_v$ may not always be applicable. Work investigating dynamic
threshold models with separate thresholds $k_v$-up and $k_v$-down is a
natural extension of this work.


\section*{Acknowledgments}

We thank Samarth Swarup, Zhengzheng Pan, and Maleq Khan for
discussions and valuable suggestions. Partial support for this work
was provided through the NSF sponsored REU program ``Modeling and
Simulation in Systems Biology'' at VBI, Virginia Tech 2010 (NSF Award
Number: 0755322). We thank REU project PI and organizer Reinhard
Laubenbacher for suggestions as well as for program coordination.

\vspace*{-3pt}   




\begin{thebibliography}{10}

\bibitem{Barrett:01g}
Christopher Barrett, Harry Hunt, Madhav Marathe, S~Ravi, Daniel Rosenkrantz,
  Richard Stearns, and Predrag Tosic.
\newblock Garden of eden and fixed points in sequential dynamical systems.
\newblock In {\em Discrete Mathematics and Theoretical Computer Science
  Proceedings}, pages 95--110, 2001.

\bibitem{Barrett:06a}
Christopher~L. Barrett, Harry~B. Hunt~{III}, Madhav~V. Marathe, S.~S. Ravi,
  Daniel~J. Rosenkrantz, and Richard~E. Stearns.
\newblock Complexity of reachability problems for finite discrete sequential
  dynamical systems.
\newblock {\em Journal of Computer and System Sciences}, 72:1317--1345, 2006.

\bibitem{Centola:09}
Damon Centola.
\newblock Failure in complex social networks.
\newblock {\em Journal of Mathematical Sociology}, 33(1):64--68, 2009.

\bibitem{Centola:07}
Damon Centola and Michael Macy.
\newblock Complex contagions and the weakness of long ties.
\newblock {\em American Journal of Sociology}, 113(3):702--734, November 2007.

\bibitem{Channakeshava:11}
K.~Channakeshava, K.~Bisset, M.~Marathe, A.~Vullikanti, and S.~Yardi.
\newblock High performance scalable and expressive modeling environment to
  study mobile malware in large dynamic networks.
\newblock In {\em Proceedings of 25th IEEE International Parallel \&
  Distributed Processing Symposium}, 2011.

\bibitem{Eubank:04}
Stephen Eubank, Hasan Glucu, V~S~Anil Kumar, Madhav~V. Marathe, Aravind
  Srinivasan, Zoltan Toroczkal, and Nan Wang.
\newblock Modeling disease outbreaks in realistic urban social networks.
\newblock {\em Nature}, 429:180--184, May 2004.

\bibitem{Goles:80}
E.~Goles and J.~Olivos.
\newblock Periodic behavior in generalized threshold functions.
\newblock {\em Discrete Mathematics}, 30:187--189, 1980.

\bibitem{Granovetter:78}
Mark Granovetter.
\newblock Threshold models of collective behavior.
\newblock {\em American Journal of Sociology}, 83(6):1420--1443, May 1978.

\bibitem{Karaoz:04}
Ulas Karaoz, T.M. Murali, Stan Letovsky, Yu~Zheng, Chunming Ding, Charles~R.
  Cantor, and Simon Kasif.
\newblock Whole-genome annotation by using evidence integration in
  functional-linkage networks.
\newblock {\em Proceedings of the National Academy of Sciences},
  101(9):2888--2893, 2004.

\bibitem{Kuhlman:12}
Chris Kuhlman, Henning~S. Mortveit, David Murrugarra, and V.~S.~Anil Kumar.
\newblock Bifurcations in {B}oolean networks.
\newblock {\em Discrete Mathematics and Theoretical Computer Science},
  AP:29--46, 2012.
\newblock Automata 2011, 21--23 November, Santiago, Chile.

\bibitem{Kumar:09}
V.~S.~Anil Kumar, Matthew Macauley, and Henning~S. Mortveit.
\newblock Limit set reachability in asynchronous graph dynamical systems.
\newblock In {\em Reachability Problems (RP) 2009}, volume 5797 of {\em Lecture
  Notes in Computer Science}, pages 217--232, Berlin/Heidelberg, 2009.
  Springer.

\bibitem{Laubenbacher:01a}
Reinhard Laubenbacher and Bodo Pareigis.
\newblock Equivalence relations on finite dynamical systems.
\newblock {\em Advances in Applied Mathematics}, 26:237--251, 2001.

\bibitem{Macauley:10b}
Matthew Macauley, Jon McCammond, and Henning~S. Mortveit.
\newblock Dynamics groups of asynchronous cellular automata.
\newblock {\em Journal of Algebraic Combinatorics}, 33(1):11--35, 2011.
\newblock Preprint: math.DS/0808.1238.

\bibitem{Macauley:11c}
Matthew Macauley and Henning~S. Mortveit.
\newblock Combinatorial characterizations of admissible coxeter sequences and
  their applications.
\newblock Accepted 2011. Preprint: math.DS/0910.4376.

\bibitem{Macauley:09a}
Matthew Macauley and Henning~S. Mortveit.
\newblock Cycle equivalence of graph dynamical systems.
\newblock {\em Nonlinearity}, 22(2):421--436, 2009.
\newblock math.DS/0709.0291.

\bibitem{Mortveit:01a}
H.~S. Mortveit and C.~M. Reidys.
\newblock Discrete, sequential dynamical systems.
\newblock {\em Discrete Mathematics}, 226:281--295, 2001.

\bibitem{Mortveit:07}
Henning~S. Mortveit and Christian Reidys.
\newblock {\em An Introduction to Sequential Dynamical Systems}.
\newblock Universitext. Springer Verlag, 2007.

\bibitem{Smith:08}
T.~Smith, N.~Maire, A.~Ross, M.~Penny, N.~Chitnis, A.~Schapira, A.~Studer,
  B.~Genton, C.~Lengeler, F.~Tediosi, D.~de~Savigny, and M.~Tanner.
\newblock Towards a comprehensive simulation model of malaria epidemiology and
  control.
\newblock {\em Parasitology}, 135:1507--1516, 2008.

\bibitem{Stanley:00}
Richard~P. Stanley.
\newblock {\em Enumerative Combinatorics: Volume 1}.
\newblock Cambridge University Press, 2000.

\end{thebibliography}
\end{document}